\documentclass
[
12pt,leqno]{amsart}
\usepackage{amsmath,amsfonts,amsthm,amssymb,dsfont}
\usepackage[alphabetic]{amsrefs}
\usepackage[OT4]{fontenc}
\usepackage{enumerate}
\usepackage{ifthen}
\newcommand{\showcomments}{yes}
\renewcommand{\showcomments}{no}

\long\def\symbolfootnote[#1]#2{\begingroup%
\def\thefootnote{\fnsymbol{footnote}}\footnote[#1]{#2}\endgroup}

\newtheorem{theorem}{Theorem}[section]
\newtheorem{lemma}[theorem]{Lemma}
\newtheorem{lem}[theorem]{Lemma}
\newtheorem{thm}[theorem]{Theorem}

\newtheorem{prop}[theorem]{Proposition}

\newtheorem{cor}[theorem]{Corollary}

\theoremstyle{definition}

\newtheorem{defin}[theorem]{Definition}

\newcommand{\Z}{\mathbb{Z}}

\newsavebox{\commentbox}
{\ifthenelse{\equal{\showcomments}{yes}}%
{\footnotemark
        \begin{lrbox}{\commentbox}
        \begin{minipage}[t]{1.25in}\raggedright\sffamily\tiny
        \footnotemark[\arabic{footnote}]}
{\begin{lrbox}{\commentbox}}}%
{\ifthenelse{\equal{\showcomments}{yes}}%
{\end{minipage}\end{lrbox}\marginpar{\usebox{\commentbox}}}
{\end{lrbox}}}

\begin{document}

\title{Separability of embedded surfaces in $3$--manifolds}

\author[P.~Przytycki]{Piotr Przytycki$^\dag$}
\address{Institute of Mathematics, Polish Academy of Sciences\\
 \'Sniadeckich 8, 00-956 Warsaw, Poland\\}
\email{pprzytyc@mimuw.edu.pl}
\thanks{$\dag$ Partially supported by MNiSW grant N201 012 32/0718 and the Foundation for Polish Science.}

\author[D.~T.~Wise]{Daniel T. Wise$^\ddag$}
           \address{Dept. of Math. \& Stats.\\
                    McGill Univ. \\
                    Montreal, QC, Canada H3A 0B9 }
           \email{wise@math.mcgill.ca}
\thanks{$\ddag$ Supported by NSERC}

\date{\today}

\maketitle

\begin{abstract}
We prove that if $S$ is a properly embedded $\pi_1$--injective surface
in a compact $3$--manifold $M$, then $\pi_1S$ is separable in $\pi_1M$.
\end{abstract}

\section{Introduction}

A subgroup $H\subset G$ is \emph{separable} if $H$ equals the intersection of finite index subgroups of $G$ containing $H$.
Scott proved that if $G=\pi_1M$ for a manifold $M$ with universal cover $\widetilde{M}$, then $H$ is separable if and only if each compact subset of $H\backslash \widetilde{M}$ embeds in an intermediate finite cover of $M$
\cite[Lem 1.4]{Sc}.
Thus, if $H=\pi_1S$ for a compact surface $S\subset H\backslash \widetilde{M}$,
then separability of $H$ implies that $S$ embeds in a finite cover of $M$.
Rubinstein--Wang found a properly immersed $\pi_1$--injective surface $S\looparrowright M$ in a graph manifold  such that $S$ does not lift to an embedding in a finite cover of $M$,
and they deduced that $\pi_1S\subset \pi_1M$ is not separable \cite[Ex 2.6]{RW}.

The objective of this paper is to prove:

\begin{thm}
\label{thm:main}
Let $M$ be a compact connected $3$--manifold and let $S\subset M$ be a properly embedded connected $\pi_1$--injective surface. Then $\pi_1S$ is separable in $\pi_1M$.
\end{thm}

The problem of separability of an embedded surface subgroup was raised for instance by Silver--Williams ---
see \cite{SWil} and the references therein to their earlier works.
The Silver--Williams conjecture was resolved recently by Friedl--Vidussi in
\cite{FV}, who proved that $\pi_1S$ can be separated from
some element in $[\pi_1M,\pi_1M]-~\pi_1S$ whenever $\pi_1S$ is not
a fiber.

We proved Theorem~\ref{thm:main} when $M$ is a graph manifold in \cite[Thm 1.1]{PW1}.
Theorem~\ref{thm:main} was also proven when $M$ is hyperbolic \cite{Hier}.
In fact, every finitely generated subgroup of $\pi_1 M$ is separable
for hyperbolic $M$, by \cite{Hier} in the case $\partial M\neq \emptyset$ and by Agol's theorem \cite{Ahak} for $M$ closed.

\medskip

\textbf{Overview:}
In Section~\ref{sec:reduction} we introduce the basic notation and reduce to studying irreducible
$M$ that is \emph{simple} in the sense that its Seifert-fibred components are products with base surfaces of sufficient complexity.
In Section~\ref{sec:high} we prove a topological result establishing separability
of finite \emph{semicovers} of $M$, i.e.\ maps required to be covers only over the interior of the blocks of the JSJ decomposition.
This requires an omnipotence result for hyperbolic manifolds with boundary \cite[Cor 16.15]{Hier} coming from virtual specialness.

To prove Theorem~\ref{thm:main} we enhance the strategy employed in \cite[Thm 1.1]{PW1} for graph manifolds.
Its main element was \cite[Constr 4.13]{PW1} which produced \emph{$S$--injective} covers of $M^g$, which are covers $\overline {M^g}$ to which $S$ lifts
and, among other properties, such that the intersection with $S$ is connected for each JSJ torus or component of $\overline {M^g}$.
We extend the construction of $S$--injective semicovers to all compact $3$--manifolds in Section~\ref{sec:injective}. We use the double coset
separability of  relatively quasiconvex subgroups of $\pi_1$ of hyperbolic $3$--manifolds with boundary \cite[Thm 16.23]{Hier} and
separability of double cosets of embedded surface subgroups
of $\pi_1$ of graph manifolds \cite[Thm 1.2]{PW1}.

We conclude with the proof of Theorem~\ref{thm:main} in Section~\ref{sec:separability}.
\medskip

\textbf{Acknowledgement:}
We thank Henry Wilton for discussions.

\section{Framework and Reductions} \label{sec:reduction}

{\bf Separability:}
We have the following finite index maneuverability: If $[H:H']<\infty$ and $H'\subset G$ is separable, then $H\subset G$ is separable. Moreover, if $[G:G']<\infty$ then a subgroup $H'\subset G'$ is separable if and only if $H'\subset G$ is separable. Finally, $H\subset G$ is separable if and only if for each $g\in G-H$ there is a finite quotient $\phi\colon G\rightarrow F$ with $\phi(g)\notin \phi(H)$. Thus $G$ is \emph{residually finite} when $\{1_G\}$ is separable. We will freely employ these statements.

\medskip

{\bf Assumptions on $M$ and $S$:}
Throughout this article $M$ is a compact connected $3$--manifold and might have nonempty boundary.
We will make additional assumptions arising from the following reductions:

We can assume that $S$ is not a sphere or a disc, since otherwise Theorem~\ref{thm:main} follows from Hempel's residual finiteness of Haken $3$--manifolds \cite{He} and Perelman's hyperbolization. By passing to a double cover we can assume that $M$ is oriented. Furthermore, if $S$ is not orientable, then the boundary $\widehat S$ of its tubular neighborhood is an oriented $\pi_1$--injective surface.
As $[\pi_1S:\pi_1\widehat S]=2$, the separability of $\pi_1\widehat S$ implies separability of $\pi_1S$. Hence we can assume that $S$ is oriented. In the presence of our assumptions, the $\pi_1$--injectivity of $S$ is equivalent to saying that $S$ is \emph{incompressible} and we will stay with this term.

\medskip

{\bf Decomposition of $M$ into blocks:}
An incompressible surface $S$ in a reducible manifold can be homotoped into one of its prime factors, say $M_0$.
Observe that there is a retraction $\pi_1M\rightarrow \pi_1M_0$ that kills the other factors. Consequently, if $g\in \pi_1M_0-\pi_1S$, and we can separate $g$ from $\pi_1S$ in a finite quotient of $\pi_1M_0$, then  we can separate $g$ from $\pi_1S$ in a finite quotient of $\pi_1M$. If $g\in \pi_1M-\pi_1M_0$, then applying \cite{He} to the factors we can find a finite cover $M'$ of $M$ where all the terms of the normal form of $g$ lie outside factor subgroups.
Then the path representing $g$ is nontrivial in the graph dual to the prime decomposition of $M'$, and it suffices to use the residual finiteness of free groups. Hence we can assume that $M$ is irreducible (though possibly $\partial$--reducible).

We will employ the \emph{JSJ decomposition} of $M$, which is the minimal collection of incompressible tori (up to isotopy) each of whose complementary
components is Seifert-fibred or atoroidal. If $M$ is a single Seifert-fibred manifold, then all finitely generated subgroups of $\pi_1M$ are separable \cite{Sc}, so we can assume that $M$ is not Seifert-fibred.

By passing to a double cover we can assume that there are no $\pi_1$--injective Klein bottles in $M$. We can also assume that $M$ is not a torus bundle over the circle, since then the only embedded surfaces are the fibers. Now a complementary component of JSJ tori cannot be simultaneously Seifert-fibred and algebraically atoroidal.
Algebraically atoroidal components are \emph{hyperbolic} by hyperbolization, in other words, their interior carries a geometrically finite hyperbolic structure (possibly of infinite volume if there are non-toroidal boundary components, as in a handlebody). We will call these
complementary components \emph{hyperbolic blocks}. The other complementary components
are Seifert-fibred and we assemble adjacent Seifert-fibred components into \emph{graph manifold blocks}. The JSJ tori that are adjacent to at least one hyperbolic block are called \emph{transitional}. An incompressible surface can be homotoped so that its intersection with each block is incompressible. Moreover, we can assume that $S$ intersects each Seifert-fibred component along a horizontal or vertical surface, unless $S$ is a $\partial$--parallel annulus.
In the latter case separability follows easily from separability of the boundary torus (since it corresponds to a maximal abelian subgroup) and from a variant of Lemma~\ref{lem:Hempel} with $T^*$ in the boundary.

\medskip

{\bf The $m$--characteristic covers and simplicity:}
For a manifold $E$ let $E_{[m]}$ denote the \emph{$m$--characteristic} cover of $E$,
which is the regular cover corresponding to the intersection of all subgroups
of index $m$ in $\pi_1E$.
In particular, if $T$ is a torus, then $T_{[m]}$ is the cover corresponding to the subgroup $m\Z\times m\Z\subset \Z\times \Z=\pi_1T$.
A Seifert-fibred manifold $E$ is \emph{simple} if it is the product of the circle with a surface of genus~$\geq 1$
that has at least $2$~boundary components.
This boundary hypothesis
(not required in \cite{PW1})
ensures that there is a retraction onto each boundary component.
Consequently, $E_{[m]}$ restricts to $m$--characteristic covers on boundary tori.
An irreducible $3$--manifold $M$ is
\emph{simple} if
its Seifert-fibred components are simple.
We will pass to a simple finite cover of $M$ in Lemma~\ref{lem:Hempel}.

\section{Extending semicovers to covers}
\label{sec:high}
We begin this section with the following additional simplification:

\begin{lem}
\label{lem:Hempel}
Let $M$ be an irreducible $3$--manifold that is not Seifert-fibred.
Then $M$ has a finite cover $M'$ that is simple. Moreover, given covers $\{T^*\}$ of the transitional tori $\{T\}$ in $M$,
we can assume that all the tori of $M'$ covering $T$ are isomorphic and factor through $T^*$.
\end{lem}

A key element of the proof employs the following omnipotence result for hyperbolic $3$--manifolds with boundary.

\begin{lem}[{\cite[Cor 16.15]{Hier}}]
\label{thm:hyp_covers}
Let $M^h$ be a hyperbolic $3$--manifold with boundary tori $\{T\}$. There exist finite covers
$\{\widehat{T}\}$ such that for any further finite covers $\{T'\}$ there
exists a finite cover ${M^h}'$ of $M^h$ that restricts on boundary tori to $\{T'\}$.
\end{lem}

By passing to a further cover we can assume that ${M^h}'\rightarrow M^h$ is regular.

\begin{proof}[Proof of Lemma~\ref{lem:Hempel}]
Luecke and Wu proved in \cite[Prop 4.4]{LW} that every graph manifold block $M^g$ of $M$ has a finite cover ${M^g}'$ that is simple. Without loss of generality
we can assume that ${M^g}'\rightarrow {M^g}$ is regular.

Choose $m$ such that
\begin{enumerate}[(i)]
\item
for any ${M^g}$ adjacent along a torus $T$ to a hyperbolic block $M^h$,
the cover $T'_{[m]}$ of the torus $T'\subset \partial {M^g}'$ covering $T$ factors
through $\widehat{T}$ of Lemma~\ref{thm:hyp_covers} and through $T^*$.
\item
for a transitional or boundary torus $T\subset M$ adjacent to a hyperbolic
block $M^h$ but not to a graph manifold block, the cover $T_{[m]}$ factors through
$\widehat{T}$ of Lemma~\ref{thm:hyp_covers} and through $T^*$, if $T$ is transitional.
\end{enumerate}

By Lemma~\ref{thm:hyp_covers}, each hyperbolic block $M^h$ of $M$ has a finite regular cover ${M^h}'$ restricting on the boundary to $\{T'_{[m]}\}$ of (i) or $\{T_{[m]}\}$ of (ii). For a Seifert-fibred component $E$ of one of the simple graph manifolds ${M^g}'$,
as $E$ is simple its retractive property ensures that
the cover $E_{[m]}$ restricts to $m$--characteristic covers on its boundary tori. Gluing appropriately many copies of the various $E_{[m]}$ and ${M^h}'$ together provides the desired
simple cover $M'$ of $M$.
\end{proof}

Henceforth we \textbf{always} assume that $M$ is simple (and
irreducible as assumed in Section~\ref{sec:reduction}).

\begin{defin}
A \emph{semicover} $\overline{M}$ of $M$ is a local embedding $\overline{M}\rightarrow M$
that restricts to a covering map over each transitional torus and over each open block.
Thus $\overline{M}$ can only fail to be a covering map at a component of $\partial\overline{M}$
that covers a transitional torus $T\subset M$. We say that $\overline{M}\rightarrow M$ is \emph{finite} if $\overline{M}$ is compact.
\end{defin}

We can now prove the main result of this section.

\begin{prop}
\label{prop:partial}
Any finite semicover $\overline M$ of $M$ has a finite cover $\overline M'\rightarrow \overline M$ that embeds in a finite cover $M'$ of $M$.
\end{prop}

\begin{proof}[Proof of Proposition~\ref{prop:partial}]
By Lemma~\ref{lem:Hempel}, there is a cover $\widehat{M}$ of $M$
such that for each transitional torus $T$ of $M$ all of the tori $\widehat{T}\subset \widehat{M}$ covering $T$ are isomorphic and factor through all the covers of $T$ in $\overline{M}$.

Let $\overline M'\rightarrow \overline M$ be the pullback of the cover $\widehat{M}\rightarrow M$ via the semicover $\overline M\rightarrow M$.
Then on all of its boundary tori the semicover $\overline{M'}\rightarrow M$ restricts to the corresponding tori in $\widehat{M}$. Gluing $\overline{M'}$ with appropriately many copies of the components of $\widehat{M}$ extends $\overline{M'}$ to a cover $M'$ of $M$.
\end{proof}

\section{Surface-injective semicovers}
\label{sec:injective}
In this section we construct a family of semicovers of $M$ to which a given surface $S\subset M$ lifts. We keep the assumptions from Section~\ref{sec:reduction}.

We will use the following case of a theorem of Mart{\'{\i}}nez-Pedroza:

\begin{thm}[{\cite[Thm 1.1]{MP}}]
\label{thm:Eduardo}
Let $S_0\subset M^h$ be an incompressible geometrically finite surface properly
embedded in a hyperbolic manifold $M^h$.
Let $\partial S_0 = C_1 \sqcup \ldots \sqcup C_k$
  and suppose these circles are contained in
boundary tori $T_1,\ldots, T_k$ of $M^h$ (some $T_i$ may coincide).
Then for almost all cyclic covers $T'_i$ of $T_i$ to which $C_i$ lift, the
graph of spaces obtained by amalgamating $S_0$ with $T'_i$ along $C_i$
maps $\pi_1$--injectively into $M^h$ and the image of its $\pi_1$ in $\pi_1M^h$ is relatively quasiconvex.
\end{thm}

The separability of double cosets of relatively quasiconvex subgroups of $\pi_1$ of a hyperbolic
$3$--manifold with boundary was established in \cite[Thm 16.23]{Hier}. Consequently, we have:

\begin{cor}
\label{cor:sep_of_quasiconvex_surfaces_with_tori}
For almost all cyclic covers $T_i'$ described in Theorem~\ref{thm:Eduardo},
the group $\pi_1(S_0\sqcup_{\{C_i\}}\{T'_i\})$ is separable
in $\pi_1M^h$.
\end{cor}

\begin{cor}
\label{cor:double_sep_of_quasiconvex}
The subgroup $\pi_1S_0$ as well as the double cosets $\pi_1S_0\pi_1T_i$ are separable in $\pi_1M^h$.
\end{cor}

\begin{defin}
Let $S\subset M$ be an incompressible surface.
A semicover $\overline{M}\rightarrow M$ to which $S$ lifts is \emph{$S$--injective} if
for each hyperbolic or graph manifold block $\overline{B}$ of $\overline{M}$ the intersection $S\cap\overline{B}$ is connected.
We allow $S$ itself to be disconnected.
\end{defin}

\begin{lemma}[{\cite[Constr 4.13]{PW1}}]
\label{lem:constr}
Let $S\subset M^g$ be a possibly disconnected incompressible surface in a graph manifold.
Suppose $n$ is an integer divisible by all of the degrees of (possibly disconnected) covers $S\cap E\rightarrow F$, where $E\subset M^g$ is a Seifert-fibred component with base surface $F$, and $S\cap E$ is horizontal. Then there is a finite cover $\overline{M^g}$ of $M^g$ to which $S$ lifts such that for each torus
$\overline T\subset \partial \overline {M^g}$ intersecting $S$:
\begin{itemize}
\item
$S\cap \overline T$ is connected,
\item
$\overline T$ maps to a torus $T\subset\partial M^g$ with degree $\frac{n}{|S\cap T|}$.
\end{itemize}
\end{lemma}

Here $|S\cap T|$ denotes the number of components
in the intersection of the surface $S$ with the torus $T$.

\begin{prop}
\label{prop:surf_injective_exist}
Let $S\subset M$ be an incompressible surface. Let $S_0$ be a component of intersection of
$S$ with a hyperbolic or graph manifold block $M_0$ of $M$. Let $T_i$ be the (possibly repeating) tori of $\partial M_0$
intersected by $S_0$. Let $g\in \pi_1 M_0-\pi_1S_0$ (resp.\  $g_i\in\pi_1M_0-\pi_1S_0\pi_1T_i$ for each $i$). Then there is
a finite $S$--injective semicover
$\overline M$ with $g\notin \pi_1\overline{M_0}$ (resp.\ $g_i\notin\pi_1\overline{M_0}\pi_1T_i$), where $\overline{M_0}$ is the block of $\overline{M}$ containing the lift of $S_0$.
\end{prop}

To make sense of the double cosets $\pi_1S_0\pi_1T_i$ inside $\pi_1M_0$, pick basepoints $x_i$ of $M_0$ in $C_i$ and interpret
$\pi_1S_0,\pi_1T_i$ as subgroups of $\pi_1M_0$ determined by loops based at $x_i$ staying in $S_0,T_i$, respectively.

\begin{proof}
In the case where we assume $g\notin \pi_1 S_0$, we use that $\pi_1 S_0$ is separable
in $\pi_1 M_0$. If $M_0$ is hyperbolic, this follows from Corollary~\ref{cor:double_sep_of_quasiconvex}. If $M_0$ is a graph manifold, we use separability of embedded surfaces in graph
manifolds \cite[Thm 1.1]{PW1}. Hence there is a finite cover $\overline{M^*_0}\rightarrow M_0$ to which $S_0$ lifts with $g\notin \pi_1\overline{M^*_0}$.

In the case where we assume $g_i\notin \pi_1S_0\pi_1T_i$ for all $i$, we use that each double coset $\pi_1S_0\pi_1T_i$ is separable
in $\pi_1M_0$. This follows from Corollary~\ref{cor:double_sep_of_quasiconvex} and \cite[Thm 1.2]{PW1}. Hence there exists a cover
$\overline{M^*_0}\rightarrow M_0$ to which $S_0$ lifts with $g_i\notin \pi_1\overline{M^*_0}\pi_1T_i$.
Let $n_i$ be the degree of the restriction of $\overline{M^*_0}\rightarrow M_0$  to the torus intersecting (the lift of) $S_0$ along (the lift of) $C_i$.

Choose $n$ so that it is divisible by the numbers in (a)--(c), and also satisfies (d):

\begin{enumerate}[(a)]
\item
every $|S\cap T|$, where $T$ is a transitional or boundary torus,
\item
 the degrees of (possibly disconnected) covers $S\cap E\rightarrow F$, where $E\subset M$ is a Seifert-fibred component
  with base surface $F$, and $S\cap E$ is horizontal,
\item
each $n_i|S\cap T_i|$ as above.
\item
We also require $\frac{n}{|S\cap T|}$ to be the degree of one of the covers $T'\rightarrow T$ given by Theorem~\ref{thm:Eduardo} for a geometrically finite component of $S\cap M^h$ in a hyperbolic block $M^h$ of $M$.
\end{enumerate}

We construct the semicover $\overline M$ in the following way.
Start with a copy $\overline S$ of $S$. Let $T$ be a transitional or boundary torus of $M$.
For each component of $S\cap T$ we attach along the corresponding circle in $\overline{S}$ the degree $\frac{n}{|S\cap T|}$ cyclic cover $\overline T$ of $T$. The value $\frac{n}{|S\cap T|}$ is an integer by~(a).

For each graph manifold block $M^g$ of $M$ consider the finite (possibly disconnected)
cover $\overline{M^g}$ from Lemma~\ref{lem:constr} applied to the surface $S\cap M^g$.
The boundary components of $\overline{M^g}$ intersecting $S$ coincide with the $\overline{T}$ attached to $\overline S$ above.

Consider now a hyperbolic block $M^h$ of $M$ such that $S\cap M^h$ is a union of fibers. In this case we choose $\overline{M^h}$ to be the
union of $|S\cap M^h|$ copies of degree $\frac{n}{|S\cap M^h|}$ cyclic covers of $M^h$ to which components of $S\cap M^h$ lift. Again,
components of $\partial\overline{M^h}$ coincide with $\overline{T}$,
so that we can consistently attach the $\overline{M^h}$ to $\overline S$.

Finally, if $S\cap M^h$ is not a union of fibers, then $\pi_1$ of each of its components is relatively quasiconvex in $\pi_1 M^h$, so
by~(d) and Corollary~\ref{cor:sep_of_quasiconvex_surfaces_with_tori}, there is a finite cover $\overline{M^h}$ extending
$S\cap M^h \cup \{\overline T\}$,
and we consistently attach the $\overline M^h$ to $\overline S$.

At this point we have constructed a finite $S$--injective semicover $\overline{M}$, without yet separating $g$ (resp.\ $g_i$). Now we replace the block $\overline{M_0}$ with its fiber product with $\overline{M^*_0}$. (Algebraically $\pi_1$ of the \emph{fiber product} is $\pi_1\overline{M_0}\cap\pi_1\overline{M^*_0}\subset \pi_1 M_0$.) This is possible by (c) which guarantees that the fiber product agrees with $\overline{M_0}$ on its boundary components intersecting $S_0$. After this replacement, $\overline{M}$ satisfies the requirement on $g$ (resp.\ $g_i$), by definition of $\overline{M^*_0}$.
\end{proof}

\section{Separability}
\label{sec:separability}
\begin{proof}[Proof of Theorem~\ref{thm:main}]
Choose a basepoint of $M$ in $S$ outside all JSJ and boundary tori. Let $f\in \pi_1M-\pi_1S$.
Consider the based cover $M^S$ of $M$ with fundamental group $\pi_1S$. Let $\gamma^S$ be a path in
$M^S$ starting at the basepoint
and representing  $f$. Then $\gamma^S$ does not terminate on $S$.
Assume that $\gamma^S$ is chosen so that it does not \emph{backtrack}, i.e.\ its image in $M$ intersects
the transitional tori a minimal number of times.

Firstly, consider the case where $\gamma^S$ terminates in a block $M^S_0\subset M^S$
that intersects the lift of $S$. Denote $S_0=S\cap M_0^S$ and let $M_0\subset M$ be the block covered by $M^S_0$. In the case where $S_0$ contains the basepoint, let $g\in \pi_1M_0$ be an element represented by a path in $M^S_0$ from the basepoint to the endpoint of $\gamma^S$.

By Proposition~\ref{prop:surf_injective_exist} there is a finite $S$--injective semicover $\overline{M}$ of $M$ with
$g\notin \pi_1 \overline{M_0}$.
Thus $\gamma^S$  projects to a path $\overline{\gamma}$ in $\overline{M}$ that
ends in $\overline{M}_0$ outside the lift of $S_0$.
By Proposition~\ref{prop:partial} the semicover $\overline{M}$ has a finite cover $\overline{M}'$ that extends to a finite
cover $M'$ of $M$. Since the endpoint of the lift of $\overline \gamma$ to $M'$, which lies
in $\overline{M}'$, does not terminate on the based connected component of the preimage of $S$, we have $f\notin \pi_1 M'\pi_1S$, as desired.

Secondly, consider the case where $\gamma^S$ terminates in a block of $M^S$ disjoint from the lift of $S$.
Let $T^S\subset M^S$ be then the first
connected component of the preimage of a transitional torus $T\subset M$ crossed by $\gamma^S$ and disjoint from $S$.
Let $M^S_0$ be the last block that $\gamma^S$ travels through before it hits
$T^S$. Let $S_0=S\cap M_0^S$ and let $M_0\subset M$ be the block covered by $M^S_0$. If $T$ coincides with one of the tori $T_i\subset M_0$ crossed by $S_0$ along $C_i$, then let $x_i\in C_i$ be a basepoint for $M_0$. Let $x_i'$ be a lift of $x_i$ in $T^S$. We keep the notation $x_i$ for the
lift of $x_i$ to $S_0\subset M^S_0$. Let $g_i\in \pi_1M_0$ be an element
represented by a path in $M^S_0$ from $x_i$ to $x_i'$.

Since $T^S$ is disjoint from $S_0$, we have $g_i\notin \pi_1 S_0 \pi_1 T_i$.
By Proposition~\ref{prop:surf_injective_exist} there is a finite $S$--injective semicover $\overline{M}$ of $M$ with
$g_i\notin \pi_1 \overline{M_0}\pi_1 T_i$ for all $i$. In other words, $\overline{\gamma}$ leaves $\overline{M_0}$ through
a torus disjoint from $S_0$.

By Proposition~\ref{prop:partial} the semicover $\overline{M}$ has a finite cover $\overline{M}'$ that extends to a finite
 cover $M'$ of $M$.
By separability of the transitional tori groups (since they are maximal abelian) and residual finiteness of the free group
(dual to transitional tori),
by replacing $M'$ with a further cover
we can assume that the lift of $\overline \gamma$ to $M'$ does not pass twice through the same transitional torus.

Let $T'\subset M'$ be the projection of $T^S$. Consider the double cover $M''$ obtained by taking two copies of
$M'$, cutting along $T'$, and regluing. Then the based connected component of the preimage of $S$ lies in one copy of (the cut) $M'$ in $M''$, while
the endpoint of the lift of $\overline\gamma$ lies in the other copy. Hence $f\notin \pi_1 M''\pi_1S$, as desired.
\end{proof}


\begin{bibdiv}
\begin{biblist}
\bib{Ahak}{article}{
   author={Agol, Ian},
   title={The virtual Haken conjecture}
   contribution={ type= {an Appendix}
                  author={Agol, Ian}
                  author={Groves, Daniel}
                  author={Manning, Jason}}
   status={preprint},
   date={2012},
   eprint={arXiv:1204.2810}}

\bib{HW}{article}{
   author={Haglund, Fr{\'e}d{\'e}ric},
   author={Wise, Daniel T.},
   title={Special cube complexes},
   journal={Geom. Funct. Anal.},
   volume={17},
   date={2008},
   number={5},
   pages={1551--1620}}


\bib{He}{article}{
   author={Hempel, John},
   title={Residual finiteness for $3$-manifolds},
   conference={
      title={Combinatorial group theory and topology},
      address={Alta, Utah},
      date={1984},
   },
   book={
      series={Ann. of Math. Stud.},
      volume={111},
      publisher={Princeton Univ. Press},
      place={Princeton, NJ},
   },
   date={1987},
   pages={379--396}
   }

\bib{FV}{article}{
   author={Friedl, Stefan},
   author={Vidussi, Stefano},
   title={A vanishing theorem for twisted Alexander polynomials with applications to symplectic $4$--manifolds},
   date={2012},
   eprint={arXiv:1205.2434}}

\bib{LW}{article}{
   author={Luecke, John},
   author={Wu, Ying-Qing},
   title={Relative Euler number and finite covers of graph manifolds},
   conference={
      title={Geometric topology},
      address={Athens, GA},
      date={1993},
   },
   book={
      series={AMS/IP Stud. Adv. Math.},
      volume={2},
      publisher={Amer. Math. Soc.},
      place={Providence, RI},
   },
   date={1997},
   pages={80--103}}

\bib{MP}{article}{
   author={Mart{\'{\i}}nez-Pedroza, Eduardo},
   title={Combination of quasiconvex subgroups of relatively hyperbolic
   groups},
   journal={Groups Geom. Dyn.},
   volume={3},
   date={2009},
   number={2},
   pages={317--342}}

\bib{PW1}{article}{
 title={Graph manifolds with boundary are virtually special},
 author={Przytycki,Piotr},
 author={Wise, Daniel T.},
 status={submitted},
 eprint={arXiv:1110.3513},
 date={2011}}




\bib{RW}{article}{
   author={Rubinstein, J. Hyam},
   author={Wang, Shicheng},
   title={$\pi_1$--injective surfaces in graph manifolds},
   journal={Comment. Math. Helv.},
   volume={73},
   date={1998},
   number={4},
   pages={499--515}}


\bib{SW}{article}{
   author={Wise, Daniel T.},
   author={Sageev, Michah},
   title={Cores for quasiconvex actions},
   date={2012},
   status={submitted},
   eprint={http://www.math.mcgill.ca/wise/papers.html}}


\bib{Sc}{article}{
   author={Scott, Peter},
   title={Subgroups of surface groups are almost geometric},
   journal={J. London Math. Soc. (2)},
   volume={17},
   date={1978},
   number={3},
   pages={555--565}}

\bib{SWil}{article}{
    AUTHOR = {Silver, Daniel S.},
    author = {Williams, Susan G.},
     TITLE = {Twisted Alexander polynomials and representation shifts},
   JOURNAL = {Bull. Lond. Math. Soc.},
    VOLUME = {41},
      date = {2009},
    NUMBER = {3},
     PAGES = {535--540}}



\bib{Hier}{article}{
   author={Wise, Daniel T.},
   title={The structure of groups with quasiconvex hierarchy},
   date={2011},
   status={submitted},
   eprint={http://www.math.mcgill.ca/wise/papers.html}}

\end{biblist}
\end{bibdiv}
\end{document}